\documentclass[12pt,a4paper]{article}
\usepackage[a4paper]{geometry}
\usepackage[utf8]{inputenc}
\usepackage{amsmath}
\usepackage{amsfonts}
\usepackage{amssymb}
\usepackage{graphicx}
\usepackage{fourier}
\usepackage{stmaryrd}
\usepackage{amsthm}
\usepackage{listings}
\usepackage{enumitem}
\usepackage{graphicx}
\usepackage{url}
\usepackage{breakurl}
\usepackage[breaklinks]{hyperref}
\usepackage{tikz}
\usetikzlibrary{matrix}

\usepackage{caption}
\usepackage[french, onelanguage, ruled, vlined, linesnumbered]{algorithm2e}

\SetCommentSty{mycommfont}

\author{Guillaume Dumas}
\title{Superpermutation Matrices}
\date{}
\newtheorem{theorem}{Theorem}
\newtheorem{proposition}{Proposition}
\newtheorem{scholia}{Scholium}
\newtheorem{lemme}{Lemma}
\newtheorem{probleme}{Problem}

\theoremstyle{definition}
\newtheorem{definition}{Definition}
\newtheorem{remark}{Remark}

\ProvideTextCommand{\textasciitilde}{OT1}{\~{}}

\begin{document}
\maketitle
\begin{abstract}
Superpermutations are words over a finite alphabet containing every permutation as a factor. Finding the minimal length of a superpermutation is still an open problem \cite{egan1, envat}. In this article, we introduce superpermutations matrices. We establish a link between the minimal size of such a matrix and the minimal length of a universal word for the quotient of the symmetric group $S_n$ by an equivalence relation. We will then give non-trivial bounds on the minimal length of such a word and prove that the limit of their ratio when $n$ approaches infinity is 2. 
\end{abstract}

\section{Introduction}
In order to motivate the introduction of superpermutation matrices, we review the recent advances on a related problem, that of the minimal superpermutation.
\subsection{Background on Superpermutations}
A \emph{word} over an alphabet $A$ is a finite sequence of letters of $A$. If $w$ is a word, $\vert w \vert$ is its number of letters and $w(i)$ its $i^{th}$ letter. Let $[n]$ denote $\lbrace 1,...,n \rbrace$. A permutation $\pi$ of $S_n$ can be represented by a word $\pi(1)...\pi(n)$ over $[n]$.\\
A \emph{$n$-superpermutation} is a universal word for permutations of $S_n$ over $[n]$: it is a word over $[n]$ that contains each permutation of $S_n$ as a contiguous factor.\par
A trivial lower bound is easy to find: to get a new permutation, we need to add at least an element to the word. We need $n!$ permutation, and we start with $n$ letters for the first, so a superpermutation has length at least $n!+n-1$.\par
Ashlock and Tillotson \cite{ashtil} conjectured in a paper in 1993 that minimal length was $n!+(n-1)!+...+1!$. For example, we have 121 for $n=2$, 123121321 for $n=3$. Their conjecture follows from a recursive construction: given a $(n-1)$-superpermutation,
\begin{itemize}
\item Write permutations of size $n-1$ in their order of appearance
\item Duplicate each permutation and add n between the copies
\item Concatenate the different parts and eliminate overlap.
\end{itemize}

If m is the length of the starting $(n-1)$-superpermutation, Ashlock and Tillotson proved that the resulting $n$-superpermutation has length $m+n!$.\par
But this conjecture was disproved later. We need first to define the transition graph $G_n$: it is a graph with $S_n$ as vertices, and the weight of an edge from $\sigma_1$ to $\sigma_2$ is the number of letters of $u$ the smallest word such that $\sigma_2$ is a suffix of $\sigma_1u$. The weight function represent the number of letters to add to a permutation to make a new permutation appear as a factor in the resulting word. Then with these definitions, we get this result:
\begin{proposition}\label{prop1}
A path H that goes through every vertex can be turned into a superpermutation u of length $w(H)+n$. A superpermutation u can be turned into a path that goes through every vertex H of weight $\vert u \vert -n$.
\end{proposition}
Notice that we don't talk about Hamiltonian paths, because we don't know if a permutation will appear only once in a minimal superpermutation.\\
With this, Houston \cite{houston1} used algorithms for the travelling salesman problem to find a $6$-superpermutation of length 872, while the conjectured length was 873.\par
The transition graph also helped provide bounds for the minimal length. An anonymous 4chan user \cite{4chan} proved the following result by studying cycles using only weight 1 edges, and cycles with weight-2 edges.
\begin{theorem}\label{th1}
A $n$-superpermutation has length at least $n!+(n-1)!+(n-2)!+n-3$.
\end{theorem}
Finally, Egan \cite{egan1} adapted a paper by Williams \cite{williams1} on the Hamiltonicity of a Cayley digraph to find a path using only weight-1 and weight-2 edges, of known weight, giving the following result.
\begin{theorem}\label{th2}
There is a superpermutation of length $n!+(n-1)!+(n-2)!+(n-3)!+n-3$.
\end{theorem}
It can be shown that it is the optimal length using only weight-1 and weight-2 edges.\par

Even though graph theory helped find good bounds, finding superpermutations of minimal length is still an open problem. But a generalisation to different alphabets can change the results \cite{envat}.

\subsection{De Bruijn words and Universal Cycles}
Let us consider an alphabet $[N]$ where $N\geq n$. We will say that two words u and v of length n are \emph{order-isomorphic} if $u(i)>u(j)\Longleftrightarrow v(i)>v(j) \, \forall 1\leq i,j\leq n$. Then any word $u$ over $[N]$ order-isomorphic to the standard representation of the permutation $\pi$ is a representation over $[N]$. So, we can consider the problem of a \emph{universal word} for permutations over $[N]$, that is to say find a word $w$ which contains a factor p order-isomorphic to $\pi$ $\forall \pi \in S_n$.\par
We have seen that finding the minimal length over $[n]$ was still an open question. But the addition of a new symbol to the alphabet changes things: over $[n+1]$ and $\mathbb{N}^*$, we can find the minimal length of a universal word for permutations.\par
The trivial lower bound $n!+n-1$ is independent of the alphabet. If we can found a universal word of this length $\mathrm{B_{inf,1}}(n)$ for all n, it would be the minimal length.\par

To do so, we need to define De Bruijn words. A \emph{De Bruijn word of order n} over an alphabet $[k]$ is a word of length $k^n$ such that every word of $[k]^n$ can be found exactly once as a cyclic factor (which means as a factor in the word $w(1)....w(k^n)w(1)w(2)...w(n-1)$). De Bruijn showed the existence of such a word \cite{bruijn}.\\

In 1992 Chung, Diaconis and Graham \cite{ucycle} looked for generalisations of this sequences for other objects, such as permutations. \emph{A universal cycle (u-cycle)} for permutations of length n is a word $w$ of length $n!$ over some alphabet $[N]$ such that every permutation is order-isomorphic to a cyclic factor of $w$.\par
If such a u-cycle exists, then $w(1)...w(n!)w(1)...w(n-1)$ is a solution for the problem of universal words over the alphabet $[N]$, and over $\mathbb{N}^*$ by inclusion.\par

\begin{theorem}\label{th3}
For all n, there is u-cycle over the alphabet $[6n]$ for permutations of length n.
\end{theorem}

Johnson \cite{johnson} showed a better result, proving that $n!+n-1$ is the minimal length over $[n+1]$.

\begin{theorem}\label{th4}
For all n, there is u-cycle over the alphabet $[n+1]$ for permutations of length n.
\end{theorem}

\section{Superpermutations Matrices}
\subsection{De Bruijn Torus and Generalisation}

Let $m,n$ be two positive integers and $A$ an alphabet. A \emph{De Bruijn Torus} for $(m,n,A)$ is a toric matrix (we identify up and down, as well as left and right) in which we can find every matrix of $\mathcal{M}_{m,n}(A)$ as block, once and only once.
In the same way than Chung, Diaconis et Graham generalised De Bruijn in universal cycles, it is natural to ask if given $F$ a set of objects that may be represented by matrices, we can determine the smallest toric matrix that contain every element of $F$ as a block. More precisely, we shall be interested in the case where $F$ is the set of permutations $S_n$.\par
\begin{definition}\label{def1}
Let $(e_1,...,e_n)$ be the standard basis of $\mathbb{R}^n$. Let $\pi\in S_n$, the associated permutation matrix is $M(\pi)=(\delta_{i,\sigma(j)})_{1\leq i,j \leq n}=(e_{\pi(1)},...,e_{\pi(n)})$. The columns of the matrix are the vectors of the base reordered according to the permutation.
\end{definition}

\begin{proposition}\label{prop2}
The map $M:\pi \mapsto M(\pi)$ is an injective group morphism into $GL_n(\mathbb{R})$. We call $S_n^{Mat}$ the image of $S_n$ by M, it is a subgroup of $GL_n(\mathbb{R})$ isomorphic to $S_n$. Hence M is a faithful representation of $S_n$.
\end{proposition}

\begin{definition}\label{def2}
Let T be a toric matrix of $\mathcal{M}_{m,p}(A)$, we say that $M\in \mathcal{M}_n(A)$ is in T as a block, and we note $M\in T$, if $\,\exists i_0\leq m, j_0\leq p$ such that $(T_{i_0+i,j_0+j})_{1\leq i,j\leq n}=M$.\\
Let E be a set of same-sized square matrices, we note $E\subset T$ if $\forall M \in E, M\in T$.
\end{definition}

\begin{definition}
A \emph{superpermutations matrices} is a  toric matrices $T\in M_{m,p}(\{0,1\})$ such that $S_n\subset T$.
\end{definition}
The question we ask ourselves is whether we can found a superpermutation matrix of minimal length. But how can we define ``of minimal size'' ? Notice that for n=1, we have $\begin{pmatrix}
1
\end{pmatrix}$ and for n=2, $\begin{pmatrix}
1 & 0\\
0 & 1
\end{pmatrix}$ that will be the smallest matrix, no matter which notion of ``size'' we consider.\\
But for n=3, we have $\begin{pmatrix}
1 & 0 & 0\\
0 & 1 & 0\\
0 & 0 & 1\\
0 & 1 & 0
\end{pmatrix}\in M_{4,3}(\{0,1\})$ and $\begin{pmatrix}
1 & 0 & 0 & 0\\
0 & 1 & 0 & 1\\
0 & 0 & 1 & 0
\end{pmatrix}\in M_{3,4}(\{0,1\})$, yet no matrix of $M_{4,4}(\{0,1\})$ works, and we need to go to 5 to find a square matrix such that $S_3 \subset \begin{pmatrix}
1 & 0 & 0 & 0 & 1\\
0 & 1 & 0 & 1 & 0\\
0 & 0 & 1 & 0 & 0\\
1 & 0 & 0 & 0 & 1\\
0 & 1 & 0 & 1 & 0
\end{pmatrix}$. Furthermore, it appears difficult to see how to place the different blocks to minimize the size without trying to fix the number of rows or columns to $n$. Hence, we restrict the study to three different problems that seem interesting: minimization in terms of rows, of columns and as a square.
\vskip 1 cm
\begin{minipage}[c]{0.30\linewidth}

\begin{tikzpicture}
\tikzset{bloc/.style={fill =cyan, fill opacity=0.2}};
\matrix(M) [matrix of math nodes, left delimiter=(,
right delimiter=), column sep=2mm, row sep=2mm]
{ 1 & 0 & 0 & 0 \\
0 & 1 & 0 & 1  \\
 0 & 0 & 1 & 0  \\ };
\draw [bloc] (M-1-1.north west) rectangle (M-3-3.south east);
\node[below=1mm] at (M.south) {$123$};
\end{tikzpicture}

\vskip 1 cm

\begin{tikzpicture}
\tikzset{bloc/.style={fill =cyan,fill opacity=0.2}};
\matrix(M) [matrix of math nodes, left delimiter=(,
right delimiter=), column sep=2mm, row sep=2mm]
{ 1 & 0 & 0 & 0 \\
0 & 1 & 0 & 1  \\
 0 & 0 & 1 & 0  \\ };
\fill [bloc] (M-1-1.north west) rectangle (M-3-1.south east);
\draw (M-1-1.north west) rectangle (M-1-1.north east);
\draw (M-1-1.north east) rectangle (M-3-1.south east);
\draw (M-3-1.south west) rectangle (M-3-1.south east);
\fill [bloc] (M-1-3.north west) rectangle (M-3-4.south east);
\draw (M-1-3.north west) rectangle (M-1-4.north east);
\draw (M-1-3.north west) rectangle (M-3-3.south west);
\draw (M-3-3.south west) rectangle (M-3-4.south east);
\node[below=1mm] at (M.south) {$321$};
\end{tikzpicture}

\end{minipage}\hfill
\begin{minipage}[c]{0.30\linewidth}

\begin{tikzpicture}
\tikzset{bloc/.style={fill =cyan, fill opacity=0.2}};
\matrix(M) [matrix of math nodes, left delimiter=(,
right delimiter=), column sep=2mm, row sep=2mm]
{ 1 & 0 & 0 & 0 \\
0 & 1 & 0 & 1  \\
 0 & 0 & 1 & 0  \\ };
\fill [bloc] (M-1-1.north west) rectangle (M-1-3.south east);
\fill [bloc] (M-2-1.north west) rectangle (M-3-3.south east);
\draw (M-1-1.north west) rectangle (M-1-1.south west);
\draw (M-1-1.south west) rectangle (M-1-3.south east);
\draw (M-1-3.north east) rectangle (M-1-3.south east);
\draw (M-2-1.north west) rectangle (M-3-1.south west);
\draw (M-2-1.north west) rectangle (M-2-3.north east);
\draw (M-2-3.north east) rectangle (M-3-3.south east);
\node[below=1mm] at (M.south) {$312$};
\end{tikzpicture}

\vskip 1 cm

\begin{tikzpicture}
\tikzset{bloc/.style={fill =cyan,fill opacity=0.2}};
\matrix(M) [matrix of math nodes, left delimiter=(,
right delimiter=), column sep=2mm, row sep=2mm]
{ 1 & 0 & 0 & 0 \\
0 & 1 & 0 & 1  \\
 0 & 0 & 1 & 0  \\ };
\fill [bloc] (M-1-1.north west) rectangle (M-1-1.south east);
\draw (M-1-1.north east) rectangle (M-1-1.south east);
\draw (M-1-1.south west) rectangle (M-1-1.south east);
\fill [bloc] (M-2-1.north west) rectangle (M-3-1.south east);
\draw (M-2-1.north west) rectangle (M-2-1.north east);
\draw (M-2-1.north east) rectangle (M-3-1.south east);
\fill [bloc] (M-1-3.north west) rectangle (M-1-4.south east);
\draw (M-1-3.north west) rectangle (M-1-3.south west);
\draw (M-1-4.south east) rectangle (M-1-3.south west);
\fill [bloc] (M-2-3.north west) rectangle (M-3-4.south east);
\draw (M-2-3.north west) rectangle (M-2-4.north east);
\draw (M-2-3.north west) rectangle (M-3-3.south west);
\node[below=1mm] at (M.south) {$213$};
\end{tikzpicture}

\end{minipage}\hfill
\begin{minipage}[c]{0.30\linewidth}

\begin{tikzpicture}
\tikzset{bloc/.style={fill =cyan,fill opacity=0.2}};
\matrix(M) [matrix of math nodes, left delimiter=(,
right delimiter=), column sep=2mm, row sep=2mm]
{ 1 & 0 & 0 & 0 \\
0 & 1 & 0 & 1  \\
 0 & 0 & 1 & 0  \\ };
\fill [bloc] (M-1-1.north west) rectangle (M-2-3.south east);
\fill [bloc] (M-3-1.north west) rectangle (M-3-3.south east);
\draw (M-1-1.north west) rectangle (M-2-1.south west);
\draw (M-2-1.south west) rectangle (M-2-3.south east);
\draw (M-1-3.north east) rectangle (M-2-3.south east);
\draw (M-3-1.north west) rectangle (M-3-1.south west);
\draw (M-3-1.north west) rectangle (M-3-3.north east);
\draw (M-3-3.north east) rectangle (M-3-3.south east);
\node[below=1mm] at (M.south) {$231$};
\end{tikzpicture}

\vskip 1 cm

\begin{tikzpicture}
\tikzset{bloc/.style={fill =cyan,fill opacity=0.2}};
\matrix(M) [matrix of math nodes, left delimiter=(,
right delimiter=), column sep=2mm, row sep=2mm]
{ 1 & 0 & 0 & 0 \\
0 & 1 & 0 & 1  \\
 0 & 0 & 1 & 0  \\ };
\fill [bloc] (M-1-1.north west) rectangle (M-2-1.south east);
\draw (M-1-1.north east) -- (M-2-1.south east);
\draw (M-2-1.south west) -- (M-2-1.south east);
\fill [bloc] (M-3-1.north west) rectangle (M-3-1.south east);
\draw (M-3-1.north east) -- (M-3-1.south east);
\draw (M-3-1.north west) -- (M-3-1.north east);
\fill [bloc] (M-1-3.north west) rectangle (M-2-4.south east);
\draw (M-1-3.north west) -- (M-2-3.south west);
\draw (M-2-3.south west) -- (M-2-4.south east);
\fill [bloc] (M-3-3.north west) rectangle (M-3-4.south east);
\draw (M-3-3.north west) -- (M-3-4.north east);
\draw (M-3-3.north west) -- (M-3-3.south west);
\node[below=1mm] at (M.south) {$132$};
\end{tikzpicture}

\end{minipage}\hfill

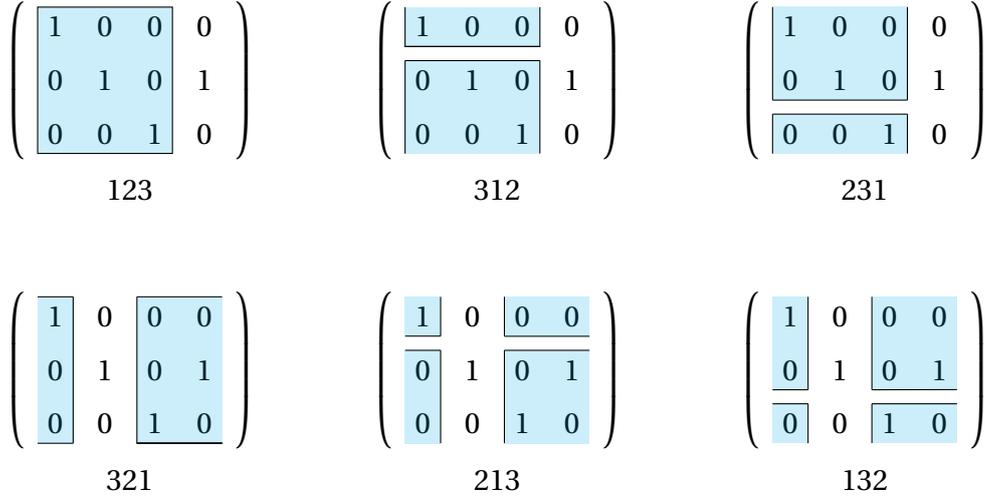
\captionof{figure}{Blocks in the $3\times 4$ matrix corresponding to the six permutations of $S_3$.}

\subsection{Definitions and equivalences between the problems}
We start by formally stating the problems we are going to focus on.
\begin{probleme}\label{pb1}
Find the smallest integer $m_1(n)$ such that $\exists T\in M_{m_1(n),n}(\{0,1\})$ with $S_n \subset T$ 
\end{probleme}

\begin{probleme}\label{pb2}
Find the smallest integer $m_2(n)$ such that $\exists T\in M_{n,m_2(n)}(\{0,1\})$ with $S_n \subset T$ 
\end{probleme}

\begin{remark}\label{r1}
In a $m\times p$ matrix, there are $mp$ distinct blocks, but we want $n!$ blocks in a u-matrix, so we have a trivial lower bound: $$m_1(n)\geq (n-1)! \textrm{ and } m_2(n)\geq (n-1)!$$
\end{remark}

\begin{probleme}\label{pb3}
Find the smallest integer $m_3(n)$ such that $\exists T\in M_{m_3(n),m_3(n)}(\{0,1\})$ with $S_n \subset T$ 
\end{probleme}
We will not study this last problem, but we can show the following result:
\begin{proposition}\label{prop3}
We have seen $m_3(1)=1$, $m_3(2)=2$ and $m_3(3)=5$. For $n\geq 5$, $m_3(n)\leq m_2(n)$.
\end{proposition}

\begin{proof}
Consider a superpermutation matrix $T$ with n rows $L_1,...,L_n$ and $m_2(n)$ columns. But $m_2(n)\geq (n-1)! \geq 2n-1$ since $n\geq 5$. We may consider the matrix $T'$ with rows $L_1,...,L_n,L_1,...,L_{n-1}$ and rows of zeros until it is squared. It is a superpermutation matrix of size $m_2(n)\times m_2(n)$.
\end{proof}

We will now show that we can study Problems \ref{pb1} or \ref{pb2} indifferently.
\begin{proposition}\label{prop4}
For all $n$, $m_1(n)=m_2(n)$
\end{proposition}
Below, we denote by $^t\!A$ the transpose of the matrix A.
\begin{lemme}\label{lem1}
If $T$ is a superpermutation matrix, then $^t\! T$ so is.\\
\end{lemme}

\begin{proof}
Let $\pi\in S_n$. T is a superpermutation matrix, hence $M(\pi^{-1})\in T$. But $\forall M \in T, ^t\! M \in ^t\! T$. But $M(\pi)=M(\pi^{-1})^{-1}=^t\!M(\pi^{-1})$ so $M(\pi)\in ^t\! T$
\end{proof}
Thus if there is a superpermutation matrix of size $n\times m$, there is one of size $m\times n$, hence $m_1(n)=m_2(n)$.

Let us now take a look at how we could solve these problems. We will now work modulo $n$, to simplify writings.
\begin{definition}\label{def3}
If $M=M(\pi)\in S_n^{Mat}$, $\mathrm{rot}(M)=\{M(\pi\sigma^i)\vert 0\leq i \leq n-1\}$ where $\sigma(i)=i+1$. If $M(\pi)=\begin{pmatrix}
e_{\pi(1)} & ...  & e_{\pi(n)}
\end{pmatrix}$ then $M(\pi)M(\sigma)=\begin{pmatrix}
e_{\pi(2)} & ...  & e_{\pi(n)} & e_{\pi(1)}
\end{pmatrix}$. We note $\mathcal{R}_1$ the equivalence relation such that $x\mathcal{R}_1 y \Longleftrightarrow y\in \mathrm{rot}(x)$.\\
If $M=M(\pi)\in S_n^{Mat}$, $\mathrm{inc}(M)=\{M(\sigma^i\pi)\vert 0\leq i \leq n-1\}$. If $M(\pi)=\begin{pmatrix}
e_{\pi(1)} & ...  & e_{\pi(n)}
\end{pmatrix}$ then $M(\sigma)M(\pi)=\begin{pmatrix}
e_{\pi(1)+1} & ...  & e_{\pi(n)+1}
\end{pmatrix}$. We can also see it this way:  if $M(\pi)=\begin{pmatrix}
L_1\\
\vdots\\
L_n
\end{pmatrix}$ then $M(\sigma)M(\pi)=\begin{pmatrix}
L_n\\
L_1\\
\vdots\\
L_{n-1}
\end{pmatrix}$. We note $\mathcal{R}_2$ the equivalence relation such that $x\mathcal{R}_2 y \Longleftrightarrow y\in \mathrm{inc}(x)$.
\end{definition}

\begin{remark}\label{r2}
$\mathrm{inc}(^t\!M)=\{^t\! N \vert N\in \mathrm{rot}(M)\}$\\
\end{remark}
By the isomorphism between $S_n$ and $S_n^{Mat}$, we will also note $\mathrm{inc}(\pi)$ and $\mathrm{rot}(\pi)$.

\begin{lemme}\label{lem2}
If $T\in \mathcal{M}_{m,n}(\{0,1\})$ is toric, then $M\in T \Rightarrow \mathrm{rot}(M)\subset T$.\\
If $T\in \mathcal{M}_{n,m}(\{0,1\})$ is toric, then $M\in T \Rightarrow \mathrm{inc}(M)\subset T$.
\end{lemme}
\begin{proof}
It follows from the remarks in Definition \ref{def3} about rows and columns and the fact that T is toric.
\end{proof}

\subsection{Transition Graph and Simplification of the Problem}
Our first approach to finding new lower and upper bounds on $m_1(n)=m_2(n)$ is to consider a transition graph.\par
Let us start with the Problem \ref{pb1} : according to Lemma \ref{lem2}, we can forget toricity on left and right and try to find $T$ which contains an element of each equivalence class for the relation $\mathcal{R}_1$. We consider the weighed complete digraph $K_n$ where the vertices are the $(n-1)!$ equivalence classes for $\mathcal{R}_1$. The weight between $\mathrm{rot}(M)$ and $\mathrm{rot}(M')$ is the minimal number of rows to add at the end of M, and remove at its beginning, to get an element of $\mathrm{rot}(M')$. For example, $w(\mathrm{rot}\begin{pmatrix}
1 & 0 & 0\\
0 & 1 & 0\\
0 & 0 & 1
\end{pmatrix},\mathrm{rot}\begin{pmatrix}
0 & 0 & 1\\
0 & 1 & 0\\
1 & 0 & 0
\end{pmatrix})=2$. Then from a path that goes through each vertex in the graph, we can take a matrix M in the class of the first vertex, then add rows according to the edges encountered, to get a matrix with n element of each class, hence with toricity and Lemma \ref{lem2}, every element of $S_n^{Mat}$. The number of rows in the matrix is then the weight of the path, plus n for the initial matrix.\\
As an example, the adjacency matrix of $K_4$ is $\begin{pmatrix}
0 & 2 & 3 & 3 & 3 & 3\\
3 & 0 & 1 & 2 & 2 & 3\\
3 & 3 & 0 & 1 & 2 & 3\\
2 & 2 & 3 & 0 & 1 & 3\\
3 & 1 & 2 & 3 & 0 & 3\\
3 & 3 & 3 & 2 & 3 & 0
\end{pmatrix}$ and with the construction we obtain the following superpermutation matrix with 12 rows: $$\begin{pmatrix}
1 & 0 & 0 & 0\\
0 & 1 & 0 & 0\\
0 & 0 & 1 & 0\\
0 & 0 & 0 & 1\\
0 & 1 & 0 & 0\\
1 & 0 & 0 & 0\\
0 & 0 & 1 & 0\\
0 & 0 & 0 & 1\\
0 & 1 & 0 & 0\\
1 & 0 & 0 & 0\\
0 & 0 & 0 & 1\\
0 & 0 & 1 & 0
\end{pmatrix}$$\\

As for problem \ref{pb2}, we can do a similar construction and look for T a matrix with n rows that contains a block of each class for the relation $\mathcal{R}_2$. Then we consider the graph $H_n$ where vertices are the $(n-1)!$ equivalence classes for $\mathcal{R}_1$. The weight between $\mathrm{inc}(M)$ and $\mathrm{inc}(M')$ is the minimal number of columns to add on the right of M, and to remove on the left, to obtain an element of $\mathrm{inc}(M')$.\\ For example, $w(\mathrm{inc}\begin{pmatrix}
1 & 0 & 0\\
0 & 1 & 0\\
0 & 0 & 1
\end{pmatrix},\mathrm{inc}\begin{pmatrix}
0 & 0 & 1\\
0 & 1 & 0\\
1 & 0 & 0
\end{pmatrix})=2$. From a path that goes through each vertex in the graph, we can take a matrix M in the class of the first vertex, then add columns according to the edges encountered, to get a matrix with n element of each class, hence with toricity and Lemma \ref{lem2}, every element of $S_n^{Mat}$. The number of columns in the matrix is then the weight of the path, plus n for the initial matrix.\\
As above, the adjacency matrix of $H_4$ is $\begin{pmatrix}
0 & 2 & 3 & 3 & 3 & 3\\
3 & 0 & 1 & 3 & 2 & 2\\
3 & 3 & 0 & 2 & 1 & 3\\
3 & 1 & 3 & 0 & 2 & 3\\
2 & 2 & 3 & 1 & 0 & 3\\
3 & 3 & 3 & 3 & 2 & 0
\end{pmatrix}$ and we can obtain a superpermutation matrix with 12 columns: $$\begin{pmatrix}
1&0&0&0&0&1&0&0&0&1&0&0\\
0&1&0&0&1&0&0&0&1&0&0&0\\
0&0&1&0&0&0&1&0&0&0&0&1\\
0&0&0&1&0&0&0&1&0&0&1&0
\end{pmatrix}$$\\

\begin{remark}\label{r3}
The path that goes through every vertex of minimal weight does not correspond to the solution of the problem, because it does not use the toricity of the matrix on left and right (respectively up and down) for Problem \ref{pb2} (respectively Problem \ref{pb1}). For $n=3$, we obtain 5 with the graphs but we saw that $m_1(3)=4$. However, we claim that the bounds from these graphs are still interesting. Furthermore, even if the two problems have the same solutions with transposition, graphs are not the same. At first, it appears that $H_n$ might be simpler to study because we can write the matrix obtained from a path as column vectors $e_i$, hence as a word, just like superpermutations.
\end{remark}
This remark leads to the definition of the following problem:

\begin{probleme}\label{pb4}
Let $\mathcal{R}$ be an equivalence relation on $S_n$, we look for the minimal length $m_{\mathcal{R}}(n)$ of a universal word for the equivalence classes of $\mathcal{R}$, that is to say a word which contains at least an element of each class as a factor.\\
We will denote $m(n):=m_{\mathcal{R}_2}(n)$, and that is the value we are interested in.
\end{probleme}

\begin{proposition}\label{prop5}
We have the inequality $m(n)-(n-1)\leq m_1(n)=m_2(n)\leq m(n)$.
\end{proposition}
\begin{proof}
Given a word $w(1)...w(k)$, we can obtain a matrix $\begin{pmatrix}
e_{w(1)} & ... & e_{w(k)}
\end{pmatrix}$ hence the inequality on the right. Given a superpermutation matrix for Problem \ref{pb2} with k column, if we add the n-1 first column at the beginning, toricity is not required and we get a universal word for the equivalence relation, so we get the inequality on the left.
\end{proof}

\begin{proposition}\label{prop6}
There is a bijection between paths in $H_n$ that go through every vertex of the graph and universal words for $\mathcal{R}_2$. Given a path H and the associated word u, $\vert u \vert = w(H)+n$.
\end{proposition}

\begin{proof}
The construction of u from H is the same as the construction of the superpermutation matrix described earlier, and we have indeed $\vert u \vert = w(H)+n$. Earlier, the issue was that toricity prevented us to construct a path for some superpermutation matrix. Here, there is not this problem anymore. Given a word u, we simply number the permutation $x_1,...,x_k$ encountered. Then there is a path $(\bar{x_1},...,\bar{x_k})$ in $H_n$ and the length is $\vert u\vert - n$ because of the definition of the weight function.
\end{proof}

\section{Bounds on m(n)}
\subsection{Searching for Upper and Lower Bounds}
\begin{proposition}\label{prop7}
The bijection shown in Section 2.3 and Proposition \ref{prop6} allows us to study the graph $H_n$ and find bounds on paths instead of words.\\
Since the weight function on edges is bounded by $1$ and $n-1$, and such a path must have at least $(n-1)!-1$ edges, and there exists a path with that number of edges, we get: 
\begin{itemize}
\item $I(n)=(n-1)!+n-1$ a trivial lower bound on the length of a universal word
\item $S(n)=(n-1)(n-1)!+1$ a trivial upper bound 
\end{itemize}
We can see that $\frac{S(n)}{I(n)}\sim n$.
\end{proposition}

\begin{center}
\includegraphics[scale=0.7]{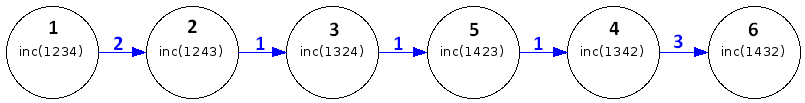}
\captionof{figure}{A path going through every vertex in $H_4$. It corresponds to the superpermutation matrix with 12 columns given in Section 2.3. The universal word corresponding to this path will be $123421342143$.}
\end{center}

We are now going to look for better bounds.
We define a 1-cycle as a cycle in $H_n$ using only weight-1 edges. We note $\textrm{cycles}^{(1)}(k)=\{C \mid C \textrm{ is a 1-cycle going through k vertices}\}$ and $\mathcal{D}(n)=\{k\in \mathbb{N}\vert k\vert n\}$.
\begin{lemme}\label{lem3}
Let $d$ be an integer. If there is a 1-cycle of length $d$, then $d\vert n$.
\end{lemme}
\begin{proof}
Let $C$ be 1-cycle, an edge goes from $\mathrm{inc}(\pi)$ to $\mathrm{inc}(\pi')$ iff $\mathrm{inc}(\pi')=\mathrm{inc}(\pi\sigma)$ with $\sigma=(12...n)$. Hence if we chose a vertex $e=\mathrm{inc}(\pi)$ in C, we can write $C=\{\mathrm{inc}(\pi\sigma^i)\vert i\in \mathbb{N}\}$. Then we can define a group structure on C where e is the identity element and $\mathrm{inc}(\pi \sigma^i)\otimes \mathrm{inc}(\pi \sigma^j)=\mathrm{inc}(\pi \sigma^{i+j})$. We get an surjective morphism from $<\sigma>\simeq \mathbb{Z}/n\mathbb{Z}$ to $C$. Then $C$ is isomorphic to a quotient group of $\mathbb{Z}/n\mathbb{Z}$ so its cardinal divides $n$.
\end{proof}

\begin{center}
\begin{tikzpicture}
\node (1) at (1,0.5) {$\mathrm{inc}(1234)$};
\node (2) at (4,2) {$\mathrm{inc}(1243)$};
\node (3) at (8,2) {$\mathrm{inc}(1324)$};
\node (4) at (4,-1) {$\mathrm{inc}(1342)$};
\node (5) at (8,-1) {$\mathrm{inc}(1423)$};
\node (6) at (11,0.5) {$\mathrm{inc}(1432)$};
\draw[->,>=latex, thick] (2) to[bend left=40] (3);
\draw[->,>=latex, thick] (3) to (5);
\draw[->,>=latex, thick] (5) to[bend left=40] (4);
\draw[->,>=latex, thick] (4) to (2);
\end{tikzpicture}
\end{center}
\captionof{figure}{1-cycles in $H_4$. We have one 1-cycle of length 4 and two 1-cycles of length 1.}

We are going to determine $\vert \mathrm{cycles}^{(1)}(d)\vert$ for $d\in \mathcal{D}(n)$.\\
Let us begin with d=1. Below, arithmetic is done modulo n.\\
\begin{lemme}\label{lem4}
For all n, the number of 1-cycles of length 1 is given by
$$\vert \mathrm{cycles}^{(1)}(1) \vert=\vert \{\pi\vert \exists k, \forall i,\pi(i+1)=\pi(i)+k \textrm{ and } \pi(1)=1\} \vert \leq n-1$$
\end{lemme}

\begin{proof}
$$\begin{tabular}{lll}
  $\mathrm{cycles}^{(1)}(1)$   &  =  & $\{\mathrm{inc}(\pi)\in S_n\vert \not \exists y, w(\mathrm{inc}(\pi),y)=1\}$\\
     & = & $\{\mathrm{inc}(\pi)\vert \mathrm{inc}(\pi)$ = $\mathrm{inc}(\pi\sigma)\}$\\
     & = & $\{\mathrm{inc}(\pi)\vert \exists \pi'\in \mathrm{inc}(\pi), \, \pi'= \pi\sigma\}$
\end{tabular}$$\\
But $\mathrm{inc}(\pi)=\{\sigma^k\pi\vert 0\leq k \leq n-1\}$ so
$$\begin{tabular}{lll}
  $\vert \mathrm{cycles}^{(1)}(1) \vert$   &  =  & $\vert \{\mathrm{inc}(\pi)\vert \exists k, \forall i, \pi\sigma(i)=\sigma^k\pi(i)\}\vert$\\
     & = & $\vert\{\mathrm{inc}(\pi)\vert \exists k, \forall i, \pi(i+1)=\pi(i)+k\} \vert$\\
     & = & $\vert \{\pi\in S_n\vert \exists k, \forall i,\pi(i+1)=\pi(i)+k \textrm{ and } \pi(1)=1\} \vert$\\
\end{tabular}$$\\
In the last equality, we fix the first value of $\pi$ in order to search for permutations instead of equivalence classes. The permutation is completely determined by $k$, so $\vert \mathrm{cycles}^{(1)}(1) \vert \leq n-1$.
\end{proof}

\begin{remark}\label{r5}
If n is a prime number, we noticed we must have $\vert 1-\mathrm{cycles}(1) \vert + n \vert 1-\mathrm{cycles}(n) \vert=(n-1)!$. According to Wilson's theorem, $(n-1)!\equiv -1 \equiv n-1 [n]$ and $\vert 1-\mathrm{cycles}(1) \vert \leq n-1$: we deduce that $\vert 1-\mathrm{cycles}(1) \vert=n-1$ and $\vert \mathrm{cycles}^{(1)}(n)\vert = \frac{(n-1)!-(n-1)}{n}$, so we completely determined the numbers of 1-cycles for n prime.\\
\end{remark}

If n is any integer, we prove the following result:
\begin{lemme}\label{lem5}
$\vert \mathrm{cycles}^{(1)}(1) \vert=\varphi(n)$ where $\varphi$ is Euler's totient function.
\end{lemme}

\begin{proof}
We are looking for $\pi$ such that $\pi(1)=1$ and $\exists k<n$, $\forall i$, $\pi(i+1)=\pi(i)+k$, which is equivalent to $\exists k<n$, $\forall i\in\mathbb{Z}/n\mathbb{Z}$, $\pi(i+1)=1+ik$. We also see that it justifies the fact that a permutation will be determined by k, so we can simply count the k that will work. We define \[
  f_k\colon\begin{aligned}[t]
        \mathbb{Z}/n\mathbb{Z}&\longrightarrow \mathbb{Z}/n\mathbb{Z}\\
        i&\longmapsto 1+ik
  \end{aligned}
\]
Then we only need to count k such that $f_k$ is bijective. $\mathbb{Z}/n\mathbb{Z}$ is finite so $f_k$ bijective $\Longleftrightarrow$ $f_k$ injective.\\
Let $i,i'$ be integers such that $1+ik=1+i'k\Longrightarrow (i'-i)k\equiv 0 [n]$.\\
\begin{itemize}
\item If $\mathcal{D}(k)\cap \mathcal{D}(n)=\{1\}$, $n\vert (i'-i)k \Longrightarrow n\vert i'-i \Longrightarrow i'=i$ in $ \mathbb{Z}/n\mathbb{Z}$
\item If $\exists p$ s.t. $n=pr_1$ and $k=pr_2$, then $i=0$ and $i'=r_1$ yields $(i'-i)k=pr_1r_2=nr_2\equiv 0 [n]$
\end{itemize}
So $f_k$ bijective $\Longleftrightarrow$ $\mathcal{D}(k)\cap \mathcal{D}(n)=\{1\}$, and thus $\vert\mathrm{cycles}^{(1)}(1) \vert=\varphi(n)$ by the definition of Euler's totient function.
\end{proof}

We can now turn to the general case, computing $\vert \mathrm{cycles}^{(1)}(d)\vert$ for any $d\in\mathcal{D}(n)$ and any n.\\
We define $$E(d,n)= \{\mathrm{inc}(\pi)\vert \exists k<n, \forall i\in  \mathbb{Z}/n\mathbb{Z},\pi(i+d)=\pi(i)+k\}$$
Notice that $\vert E(1,n)\vert = \vert \mathrm{cycles}^{(1)}(1) \vert$. The situation is however not that simple for $d\not = 1$.
\begin{proposition}\label{prop8}
$$E(d,n)=\bigcup_{k\vert d} \quad \bigcup_{C\in \mathrm{cycles}^{(1)}(k)} C$$\\
We can deduce a recursive computation: $$\vert \mathrm{cycles}^{(1)}(d)\vert = \frac{1}{d}(\vert E(d,n) \vert - \sum_{k\vert d, k\not = d} k\vert \mathrm{cycles}^{(1)}(k)\vert )$$
\end{proposition}

\begin{proof}
Let $$A(d,n)=\bigcup_{k\vert d} \quad \bigcup_{C\in \mathrm{cycles}^{(1)}(k)} C$$
We want to show that $\forall n,d\in \mathbb{N}$, $E(d,n)=A(d,n)$.\\

Let $x\in A(d,n)$, $\exists k\vert d$, there exists $C\in cycles^{(1)}(k)$ such that $x\in C$. There exists also $\pi \in S_n$ such that $x=\mathrm{inc}(\pi)$. Then, since $x$ is in a 1-cycle of length $k$, we have $\mathrm{inc}(\pi\sigma^k)=\mathrm{inc}(\pi)$. It means there is $\pi'\in \mathrm{inc}(\pi)$ such that $\pi'=\pi\sigma^k$. But $\mathrm{inc}=\{\sigma^p\vert 0\leq p < n\}$. Thus, $\exists p\leq n-1$ s.t. $\pi'=\sigma^p\pi=\pi\sigma^k$. Hence for all $i\in \mathbb{Z}/n\mathbb{Z}$, $\sigma^p\pi(i)=\pi(i)+p=\pi\sigma^k(i)=\pi(i+k)$.\\
Then since $k\vert d$, we can write $d=kt$ for some $t\in \mathbb{N}$. We can see that $$\begin{tabular}{lllr}
$\pi(i+d)$&$=$&$\pi(i+k(t-1)+k)$&\\
&$=$&$\pi(i+k(t-1))+p$&\\
&$=$&$\pi(i)+tp$&\textrm{ by induction}
\end{tabular}$$
Let $r$ be the remainder of the euclidean division of $tp$ by $n$. We get that $\exists r<n$, $\forall i\in \mathbb{Z}/n\mathbb{Z}$, $\pi(i+d)=\pi(i)+r$, which means that $x\in E(d,n)$ by definition of $E(d,n)$.\\

Let $x=\mathrm{inc}(\pi)\in E(d,n)$, and $p$ the integer such that $\forall i \in \mathbb{Z}/n\mathbb{Z}$, $\pi(i+d)=\pi(i)+p$, which means that $\pi\sigma^d=\sigma^p\pi$. The 1-cycle of $x$ is $C=\{\mathrm{inc}(\pi\sigma^i)\vert i\in \mathbb{N}\}$. If we use the same group structure as in Lemma \ref{lem3} with $\mathrm{inc}(\pi)$ as the identity element, then we can define a surjective morphism $$\varphi : \begin{tabular}{lll}
$(\mathbb{Z},+)$&$\rightarrow$& $(C,\otimes)$\\
$i$ & $\mapsto$ & $\mathrm{inc}(\pi\sigma^i)$
\end{tabular}$$
The fact that $\pi\sigma^d=\sigma^p\pi$ means that $d\mathbb{Z}\subset Ker(\varphi)$, so we have a surjective morphism from $\mathbb{Z}/d\mathbb{Z}$ into $C$. Thus, $C$ is isomorphic to a quotient group of $\mathbb{Z}/d\mathbb{Z}$, hence $\vert C\vert$ divides $d$. $x$ is in a 1-cycle of length k with $k\vert d$, so $x\in A(d,n)$.
\end{proof}

Then we only need to count the elements of $E(d,n)$.
\begin{lemme}\label{lem6}
For any $d\in\mathcal{D}(n)$, if $p:=\frac{n}{d}$, we have $\vert E(d,n) \vert = \varphi(p)p^{d-1}(d-1)!$.
\end{lemme}

\begin{proof}
As for d=1, we are going to count permutations instead of equivalence classes by stating $\pi(1)=1$.\\
First, we must search the possible values of k :\\
We have $\forall i,j\in \mathbb{Z}/n\mathbb{Z}$, $\pi(i+jd)=\pi(i)+jk$. In particular, if j=p, $\pi(i+pd)=\pi(i)+pk=\pi(i)$ so we need $n\vert pk$ hence $d\vert k$.\\
We also need \[
  f_k\colon\begin{aligned}[t]
        \mathbb{Z}/p\mathbb{Z}&\longrightarrow \mathbb{Z}/n\mathbb{Z}\\
        j&\longmapsto 1+jk
  \end{aligned}
\]
to be injective. Let $j,j'\in \mathbb{Z}/n\mathbb{Z}$ s.t. $1+jk=1+j'k$, so $(j-j')k\equiv 0[n]$. We have $k=dk'$, so $n\vert dk'(j-j')$ means $p\vert k'(j-j')$.
\begin{itemize}
\item If $k'\wedge p$, $p\vert (j-j')k' \Longrightarrow p\vert j-j' \Longrightarrow j'=j$ in $ \mathbb{Z}/p\mathbb{Z}$
\item If $\exists t$ s.t. $p=tr_1$ and $k'=tr_2$, then $j'=0$ and $j=r_1$ yields $(j-j')k'=tr_1r_2=pr_2\equiv 0 [p]$.
\end{itemize}
Finally, k works iff $d\vert k$ and $\frac{k}{d}\wedge p=1$ which amount to choosing $k'<p$ with $k'\wedge p$ hence there are $\varphi(p)$ possibilities for k.\\

Now we must count, for a given k, the number of $\pi$ that verifies the conditions (with d=1, there was only one permutation, but not anymore). Indeed, letters are fixed by groups of p : if we know $\pi(i)$, we will know $\pi(i+d)$,...,$\pi(i+(p-1)d)$. Hence we see that we only need to fix $\pi(i)$ with $i\leq d$ to completely determine $\pi$. We considered $\pi(1)=1$. There are $n-p$ possibilities left for $\pi(2)$, then $n-2p$ for $\pi(3)$ until $n-(d-1)p$ for $\pi(d)$. For a given k, the number of permutations is $$\prod_{k=1}^{d-1}(n-pk)=\prod_{k=1}^{d-1}p(d-k)=p^{d-1}(d-1)!$$\\
Multiplied by the number of k, $\vert E(d,n) \vert = \varphi(p)p^{d-1}(d-1)!$.

\end{proof}

\begin{theorem}\label{th5}
We have he following upper bound:
$$m(n)\leq B(n) = 1+ \sum_{d\vert n}\vert \mathrm{cycles}^{(1)}(d)\vert(d+n-2)$$ 
\end{theorem}

\begin{proof}
We construct Hamiltonian path $H$ the following way:
\begin{itemize}
\item Start from some vertex $x$
\item Follow weight-1 edges to reach every vertex in its 1-cycle
\item Head to the nearest 1-cycle with an edge of weight at most n-1
\item Repeat until reaching every vertex.
\end{itemize}

We can now show that the weight of $H$ is smaller than a function of $n$: each 1-cycle of cardinal $d$ requires $d-1$ weight-1 edges, and each 1-cycle other than the first one requires an edge of weight at most $n-1$. Hence $$w(H)\leq \sum_{d\vert n} \vert \mathrm{cycles}^{(1)}(d)\vert(d-1) + (\sum_{d\vert n}\vert \mathrm{cycles}^{(1)}(d)\vert -1)(n-1)\leq  \sum_{d\vert n} \vert \mathrm{cycles}^{(1)}(d)\vert (d+n-2) -(n-1) $$\\
In order to obtain a universal word, we must add the $n$ letters of the first permutation, which gives the result we wanted.
\end{proof}

\begin{theorem}\label{th6}
We also get a lower bound : $$m(n)\geq C(n)=(n-1)!+n-1+\sum_{d\vert n}\vert \mathrm{cycles}^{(1)}(d)\vert$$
\end{theorem}

\begin{proof}
Let $\mathcal{C}=(x_1,...,x_m)$ be a path, $p(\mathcal{C})$ the number of distinct vertices visited by $\mathcal{C}$ and $c(\mathcal{C})$ the number of distinct 1-cycles visited.\\
We will show the inequality $w(\mathcal{C})\geq p(\mathcal{C})+c(\mathcal{C})-2$ by induction on $m$.
\begin{itemize}
\item $m=1$: $w=0$, $p=1$ and $c=1$ so we have the inequality
\item Let $\mathcal{C'}=(x_1,...,x_{m+1})$. Suppose that $\mathcal{C}=(x_1,...,x_m)$ verifies the inequality.
\begin{itemize}
\item If $w(x_m,x_{m+1})\geq 2$, then $p$ and $c$ can only be increased by one so the inequality holds
\item Si $w(x_m,x_{m+1}) = 1$, we stay in the same 1-cycle so $c$ does not increase, hence the inequality holds
\end{itemize}
But if $\mathcal{C}$ goes through every vertex, $p(\mathcal{C})=(n-1)!$ and $c(\mathcal{C})=\sum_{d\vert n}\mathrm{cycles}^{(1)}(d)$.\\ $w(\mathcal{C})\geq (n-1)!-2+\sum_{d\vert n}\vert \mathrm{cycles}^{(1)}(d)\vert$ and then for every universal word (by Prop. \ref{prop7}), we get $\vert u \vert \geq (n-1)!+n-2+\sum_{d\vert n}\vert \mathrm{cycles}^{(1)}(d)\vert$.
\end{itemize}
\end{proof}

\subsection{Tightness of the Bounds}
Let us take a look at the first values of n:
\begin{center}
\begin{tabular}{|c|c|c|c|c|c|}
\hline
n & I(n) & C(n) & Best found & B(n) & S(n)\\
\hline
1&1&1&1&1&1\\
2&2&2&2&3&2\\
3&4&5&5&5&5\\
4&9&11&12&13&19\\
5&28&35&39&49&97\\
6&125&148&164&217&601\\
7&726&823&915&1261&4321\\
8&5047&5686&6118&8881&35280\\

\hline
\end{tabular}
\captionof{figure}{Table of the values of the bounds for $n\leq 8$. ''Best found`` is the minimal size found numerically for $m(n)$. It is optimal at least for $n\leq 4$.}
\end{center}
We saw in Proposition \ref{prop7} that the ratio between trivial upper and lower bounds was equivalent to $n$. In order to show that we improved significantly the bounds, we would like to show that the limit of the ratio between the new upper bound and a lower bound when $n$ approaches infinity is a constant.

\begin{proposition}\label{prop9}
If n is prime, the upper bound can be rewritten: $$\begin{tabular}{lll}

$B(n)$& $=$& $1 + (n-1)\vert \mathrm{cycles}^{(1)}(1)\vert   +(2n-2)\vert \mathrm{cycles}^{(1)}(n)\vert$\\
 &$=$ &$1+(n-1)^2+2(n-1)[\frac{(n-1)!-(n-1}{n}]$\\
 &$=$&$1+(n-1)^2(2\frac{(n-2)!-1}{n}+1)$
\end{tabular}$$

If we denote by $p_n$ the $n^{th}$ prime number, then $$\frac{B(p_n)}{(p_n-1)!} \xrightarrow[n\rightarrow +\infty]{} 2$$
\end{proposition}

We note $u_n=\frac{B(n)}{(n-1)!}$, the is a subsequence of $u_n$ whose limit is $2$. Using a computer, it seems that $u_n$ grows towards $2$ after $n=6$. If we can indeed prove that $u_n$ is increasing, we will get the limit we want since $u_n$ converges towards $l\in \bar{\mathbb{R}}$, but $l$ has to be $2$, the limit of the subsequence. We will even get $u_n\leq 2(n-1)!$ for $n\geq 6$.\\
Unfortunately, I could not prove this result, but I was able to prove the limit of the ratio using a second bound $B'(n)$, more explicit but worse, and show that $\frac{B'(n)}{(n-1)!}\xrightarrow[n\rightarrow \infty]{}2$.

\begin{proposition}\label{prop10}
Let $$B'(n)=1+(n-1)!+\left(\sum_{d\vert n}\varphi(\frac{n}{d})\frac{n^{d-1}(d-1)!}{d^d}\right)(n-2)$$
We have $m(n)\leq B(n) \leq B'(n)$.
\end{proposition}

\begin{proof}
Remember that $$B(n)=1+\sum_{d\vert n}\vert \mathrm{cycles}^{(1)}(d)\vert(d+n-2)$$ \textrm{ and } $$E(d,n)=\sum_{k\vert d}\vert \mathrm{cycles}^{(1)}(k)\vert k=\varphi(\frac{n}{d})(\frac{n}{d})^{d-1}(d-1)!$$

We see that $\vert \mathrm{cycles}^{(1)}(d)\vert \leq \frac{E(d,n)}{d}$. 
Then we can write
$$\begin{tabular}{ccc}
     $B(n)$ & $=$ & $1+\sum_{d\vert n}\vert \mathrm{cycles}^{(1)}(d)\vert d+ (\sum_{d\vert n}\vert 1-\mathrm{cycles}(d)\vert)(n-2)$ \\
     & $=$ & $1+E(n,n)+(\sum_{d\vert n}\vert \mathrm{cycles}^{(1)}(d)\vert)(n-2)$\\
     & $=$ & $1+(n-1)!+(\sum_{d\vert n}\vert \mathrm{cycles}^{(1)}(d)\vert)(n-2)$\\
     & $\leq$ & $1+(n-1)!+(\sum_{d\vert n}\frac{E(d,n)}{d})(n-2)$\\
     & $\leq$ & $1+(n-1)!+(\sum_{d\vert n}\varphi(\frac{n}{d})\frac{n^{d-1}(d-1)!}{d^d})(n-2)$
\end{tabular}$$
\end{proof}

\begin{scholia}\label{s1}
We can write $B(n)=1+(n-1)!+(\sum_{d\vert n}\vert \mathrm{cycles}^{(1)}(d)\vert)(n-2)$ which is better for understanding, and can be compared with the lower bound.
\end{scholia}

\begin{theorem}\label{th7}
$$\frac{B'(n)}{(n-1)!}\xrightarrow[n\rightarrow \infty]{}2$$
\end{theorem}

\begin{proof}
$$\frac{B'(n)}{(n-1)!}=\frac{1}{(n-1)!}+1+\left(\sum_{d\vert n}\varphi\left(\frac{n}{d}\right)\frac{n^{d-1}(d-1)!}{d^d(n-2)!}\right)\frac{n-2}{n-1}$$
It is sufficient to show that $$\sum_{d\vert n}\varphi\left(\frac{n}{d}\right)\frac{n^{d-1}(d-1)!}{d^d(n-2)!}=\sum_{d\vert n,d\not = n}\varphi\left(\frac{n}{d}\right)\frac{n^{d-1}(d-1)!}{d^d(n-2)!}+\frac{(n-1)!}{n(n-2)!}\xrightarrow[n\rightarrow \infty]{}1$$
hence that $$v_n=\sum_{d\vert n,d\not = n}\varphi\left(\frac{n}{d}\right)\frac{n^{d-1}(d-1)!}{d^d(n-2)!}\xrightarrow[n\rightarrow \infty]{}0$$
Consider any term in this sum. $$\varphi\left(\frac{n}{d}\right)\frac{n^{d-1}(d-1)!}{d^d(n-2)!}\leq \left(\frac{n}{d}\right)^d\frac{n(n-1)}{d^2}\frac{d!}{n!}$$
$$\frac{d!}{n!} = \frac{1}{(n-d)!}\prod_{i=1}^d\frac{i}{n-d+i}\leq \frac{1}{(n-d)!}\left(\frac{d}{n}\right)^d$$
since $$\begin{tabular}{ccc}
     $i \leq d$ & $\Longleftrightarrow$ & $(n-d)i\leq (n-d)d $ \\
     & $\Longleftrightarrow$ & $ni\leq di+nd+d^2$\\
     & $\Longleftrightarrow$ & $\frac{i}{n-d+i}\leq \frac{d}{n}$
\end{tabular}$$

then $$\varphi\left(\frac{n}{d}\right)\frac{n^{d-1}(d-1)!}{d^d(n-2)!}\leq \frac{n(n-1)}{d^2(n-d)!}$$ with $d\vert n$ and $d\not = n$ so $d\leq \frac{n}{2}$, which means $(n-d)!\geq \left(\lfloor \frac{n}{2} \rfloor \right)!$ hence $$\varphi\left(\frac{n}{d}\right)\frac{n^{d-1}(d-1)!}{d^d(n-2)!}\leq \frac{n(n-1)}{d^2\left(\lfloor \frac{n}{2} \rfloor \right)!}$$
Summing over the values of d, $$v_n=\sum_{d\vert n,d\not = n}\varphi\left(\frac{n}{d}\right)\frac{n^{d-1}(d-1)!}{d^d(n-2)!}\leq \sum_{d\vert n,d\not = n} \frac{n(n-1)}{d^2\left(\lfloor \frac{n}{2} \rfloor \right)!}\leq \frac{n^3}{\left(\lfloor \frac{n}{2} \rfloor \right)!} \xrightarrow[n\rightarrow \infty]{}0$$
\end{proof}

\begin{lemme}\label{lem7}
Let $L(n)=\sum_{d\vert n}\vert \mathrm{cycles}^{(1)}(d)\vert$. Then $\frac{L(n)}{(n-2)!}\xrightarrow[n\rightarrow \infty]{} 1$
\end{lemme}

\begin{proof}
The worst and best cases are respectively: only 1-cycles of length 1 and only 1-cycles of length n, which gives us $\frac{(n-1)!}{n}\leq L(n) \leq (n-1)!$ so $\frac{(n-1)}{n}\leq \frac{L(n)}{(n-2)!} \leq (n-1)$.\\
On the left, the limit is 1 but on the right it is $\infty$. However, we see as a scholium of Theorem \ref{th7} and Proposition \ref{prop10} that $$\dfrac{L(n)}{(n-2)!}\leq \sum_{d\vert n}\varphi\left(\frac{n}{d}\right)\frac{n^{d-1}(d-1)!}{d^d(n-2)!}\xrightarrow[n\rightarrow \infty]{}1$$
The squeeze theorem gives the result.
\end{proof}

\begin{theorem}\label{th8}
$$\frac{B(n)}{(n-1)!}\xrightarrow[n\rightarrow \infty]{}2$$
\end{theorem}

\begin{proof}
Scholium \ref{s1} says that $\frac{B(n)}{(n-1)!}=1+\frac{1}{(n-1)!}+\frac{L(n)(n-2)}{(n-1)!}$. The previous lemma gives the limit we want.
\end{proof}

\section{Conclusion}
In this article, we gave a definition of a new problem on matrices, close to superpermutations. By narrowing this problem, we were able to come back to finding a universal word instead of a matrix, and then use tools on transition graph.\par
We managed to find upper and lower bounds: algorithmically, we can construct a path verifying the upper bound simply by visiting the nearest neighbour at each step, but of course creating the graph requires a lot of time.\par
Also, the bound itself is hard to compute because it requires recursive computation. But we were able to find a more explicit bound, worse, but easier to compute. We even showed that both bounds were equivalent to $2I(n)$.\par

Here are some open questions:
\begin{itemize}
\item Is the word of minimal length unique up to relabelling ?
\item Can we prove that $u_n$ is increasing ?
\item Can we find better bounds by looking at weight-2 edges ? It seems harder, because unlike weight-1 edges, weight-2 edges are not unique.
\item Can we found bounds such that the ratio between the upper bound and the lower bound approaches 1, just like superpermutations ?
\item Is there a link between the minimal length of a superpermutation and $m(n)$ ?
\item Are there bounds for other problems on superpermutation matrices ? For example, Problem \ref{pb3} is about square matrices, but I also wondered if it was possible to minimize the area or the density of 1.
\end{itemize}

\section*{Acknowledgments}
I wish to thank Ion Nechita, who supervised my internship and helped me a lot to process my ideas, as well as the LPT Toulouse that welcomed me. I also thank Baptiste Boisan who helped me for a calculation.

\nocite{*}
\bibliographystyle{plain}
\bibliography{bibli}

\begin{thebibliography}{1}

\bibitem{4chan}
Anonymous 4chan user, R.~Houston, J.~Pantone, and V.~Vatter.
\newblock A lower bound on the length of the shortest superpattern.
\newblock \url{https://oeis.org/A180632/a180632.pdf}, 2018.

\bibitem{ashtil}
D.~A. Ashlock and J.~Tillotson.
\newblock Construction of small superpermutations and minimal injective
  superstrings.
\newblock {\em Congressus Numerentium}, 93:91--98, 1993.

\bibitem{ucycle}
F.~Chung, P.~Diaconis, and R.~Graham.
\newblock Universal cycles for combinatorial problems.
\newblock {\em Discrete Mathematics}, 110:43--59, 1992.

\bibitem{bruijn}
N.~G. De~Bruijn.
\newblock A combinatorial problem.
\newblock {\em Proceedings of the Section of Sciences of the Koninklijke
  Nederlandse Akademie van Wetenschappen te Amsterdam}, 49:758--764, 1946.

\bibitem{egan1}
G.~{Egan}.
\newblock Superpermutations.
\newblock
  \url{http://www.gregegan.net/SCIENCE/Superpermutations/Superpermutations.html},
  2018.

\bibitem{envat}
M.~{Engen} and V.~{Vatter}.
\newblock {Containing all permutations}.
\newblock {\em arXiv e-prints}, page arXiv:1810.08252, Oct 2018.

\bibitem{houston1}
R.~{Houston}.
\newblock {Tackling the Minimal Superpermutation Problem}.
\newblock {\em arXiv e-prints}, page arXiv:1408.5108, Aug 2014.

\bibitem{johnson}
J.~R. {Johnson}.
\newblock {Universal cycles for permutations}.
\newblock {\em arXiv e-prints}, page arXiv:0710.5611, Oct 2007.

\bibitem{williams1}
A.~{Williams}.
\newblock {Hamiltonicity of the Cayley Digraph on the Symmetric Group Generated
  by sigma = (1 2 ... n) and tau = (1 2)}.
\newblock {\em arXiv e-prints}, page arXiv:1307.2549, Jul 2013.

\end{thebibliography}

\end{document}